\documentclass{amsart} 

\usepackage{lmodern} 
\usepackage{microtype} 
\usepackage[english]{babel} 
\usepackage[bookmarksdepth=1]{hyperref}

\theoremstyle{plain} 
\newtheorem*{bohr}{Bohr's theorem} 
\newtheorem{theorem}{Theorem} 

\newcommand{\sigmac}{\sigma_{\mathrm{c}}} 
\newcommand{\sigmau}{\sigma_{\mathrm{u}}}

\begin{document} 
\title{An extension of Bohr's theorem} 
\date{\today} 

\author{Ole Fredrik Brevig} 
\address{Department of Mathematics, University of Oslo, 0851 Oslo, Norway} 
\email{obrevig@math.uio.no}

\author{Athanasios Kouroupis} 
\address{Department of Mathematical Sciences, Norwegian University of Science and Technology (NTNU), 7491 Trondheim, Norway} 
\email{athanasios.kouroupis@ntnu.no}
\begin{abstract}
	The following extension of Bohr's theorem is established: If a somewhere convergent Dirichlet series $f$ has an analytic continuation to the half-plane $\mathbb{C}_\theta = \{s = \sigma+it\,:\, \sigma>\theta\}$ that maps $\mathbb{C}_\theta$ to $\mathbb{C} \setminus \{\alpha,\beta\}$ for complex numbers $\alpha \neq \beta$, then $f$ converges uniformly in $\mathbb{C}_{\theta+\varepsilon}$ for any $\varepsilon>0$. The extension is optimal in the sense that the assertion no longer holds should $\mathbb{C}\setminus \{\alpha,\beta\}$ be replaced with $\mathbb{C}\setminus \{\alpha\}$. 
\end{abstract}


\maketitle

\section{Introduction}
Let $\mathfrak{D}$ denote the class of Dirichlet series 
\begin{equation}\label{eq:diriseri} 
	f(s) = \sum_{n=1}^\infty a_n n^{-s} 
\end{equation}
that converge in at least one point $s=\sigma+it$ in the complex plane. Associated to each Dirichlet series $f$ in $\mathfrak{D}$ is a number $\sigmac(f)$, called the \emph{abscissa of convergence}, with the property that $f$ converges if $\sigma>\sigmac(f)$ and $f$ does not converge if $\sigma<\sigmac(f)$. This note concerns an extension of Bohr's classical theorem on uniform convergence of Dirichlet series~\cite{Bohr1913}. We therefore define the \emph{abscissa of uniform convergence} $\sigmau(f)$ as the infimum of the real numbers $\theta$ such that $f$ converges uniformly in the half-plane $\mathbb{C}_\theta$. Here and in what follows, we set
\[\mathbb{C}_\theta = \{s = \sigma+it\,:\, \sigma>\theta\}.\]

Our starting point reads as follows. 
\begin{bohr}
	Let $f$ be in $\mathfrak{D}$. If there is a real number $\theta$ and a bounded set $\Omega$ such that $f$ has an analytic continuation to $\mathbb{C}_\theta$ that maps $\mathbb{C}_\theta$ to $\Omega$, then $\sigmau(f) \leq \theta$. 
\end{bohr}

Queff\'{e}lec and Seip~\cite{QS2015} (see also~\cite[Theorem~8.4.1]{QQ2020}) showed that the assumption that $\Omega$ is a bounded set may be replaced with the weaker assumption that $\Omega$ is a half-plane. This extension of Bohr's theorem was applied obtain the canonical formulation of the Gordon--Hedenmalm characterization of composition operators~\cite{GH1999}, which has proven to be essential for further developments (see e.g.~\cite[Section~6]{BP2021}).

The purpose of the present note is to delineate precisely the limits to how far Bohr's theorem may be extended in terms of the mapping properties of $f$ in the half-plane $\mathbb{C}_\theta$. We will achieve this by establishing the following results. 
\begin{theorem}\label{thm:bohrext} 
	Let $f$ be in $\mathfrak{D}$. If there is a real number $\theta$ and complex numbers $\alpha\neq\beta$ such that $f$ has an analytic continuation to $\mathbb{C}_\theta$ that maps $\mathbb{C}_\theta$ to $\mathbb{C}\setminus\{\alpha,\beta\}$, then $\sigmau(f)\leq \theta$. 
\end{theorem}
\begin{theorem}\label{thm:nomoreext} 
	There is a Dirichlet series $f$ with $\sigmac(f) \leq 1/2$, $\sigmau(f) = 1$, and
	\[f(\mathbb{C}_\theta) = \mathbb{C}\setminus \{0\}\]
	for any $1/2 \leq \theta \leq 1$.
\end{theorem}

It must be stressed that both results are fairly direct consequences of well-known techniques and results. The proof of Theorem~\ref{thm:bohrext} uses Schottsky's theorem similarly to how it is used by Titchmarsh in the introduction to \cite[Chapter~XI]{Titchmarsh1986}, while Theorem~\ref{thm:nomoreext} is deduced from results of Bohr \cite{Bohr1911} and Helson \cite{Helson1969} on the Riemann zeta function.

\subsection*{Acknowledgements.} The authors thank Herv\'{e} Queff\'{e}lec for providing helpful comments.

\section{Proof of Theorem~\ref{thm:bohrext} and Theorem~\ref{thm:nomoreext}} 
We begin with some preparation for the proof of Theorem~\ref{thm:bohrext}. Let $\mathbb{D}(c,r)$ denote the open disc with center $c$ and radius $r>0$. If $f$ is analytic and different from $0$ and $1$ in $\mathbb{D}(c,r)$, then the effective version of Schottsky's theorem due to Ahlfors \cite{Ahlfors1938} states that 
\begin{equation}\label{eq:schottsky} 
	|f(s)| \leq \exp\left(\frac{r+|s-c|}{r-|s-c|}\big(7+\max(0,\log|f(c)|)\big)\right) 
\end{equation}
for all $s$ in $\mathbb{D}(c,r)$. (We do not actually require the effective version of Schottsky's theorem, but we find it more convenient to work with explicit expressions.)
\begin{proof}[Proof of Theorem~\ref{thm:bohrext}] 
	We may assume without loss of generality that $\alpha=0$ and $\beta=1$. It is well-known (see e.g.~\cite[Chapter~4.2]{QQ2020}) that $\sigmau(f) \leq \sigmac(f) + 1$, so every Dirichlet series in $\mathfrak{D}$ converges uniformly in some half-plane. For $\vartheta>\theta$, we set
	\[M(f,\vartheta) = \sup_{t \in \mathbb{R}} |f(\vartheta+it)|.\]
	It is plain that $M(f,\vartheta)<\infty$ if $\vartheta>\sigmau(f)$. We fix $\vartheta>\sigmau(f)$ and apply \eqref{eq:schottsky} with $c = \vartheta+it$, $r=\vartheta-\theta$, and $s = \sigma+it$, to infer that if $\theta<\sigma<\vartheta$, then 
	\begin{equation}\label{eq:schottsky2} 
		|f(s)| \leq \exp\left(\frac{2(\vartheta-\theta)}{\sigma-\theta}\big(7+\max(0,\log|M(f,\vartheta)|)\big)\right). 
	\end{equation}
	This demonstrates that $f$ is bounded in $\mathbb{C}_{\theta+\varepsilon}$ for any $\varepsilon>0$, and, consequently, that $\sigmau(f) \leq \theta$ by Bohr's theorem. 
\end{proof}

Ritt \cite[Theorem~II]{Ritt1928} established a version of Schottsky's theorem for convergent Dirichlet series. This result provides an upper bound similar to \eqref{eq:schottsky2} that is valid in all of $\mathbb{C}_\theta$ and that only depends on $\theta$ and $a_1$, under the additional assumption that $a_1$ is not equal to $0$ or $1$. Here $a_1$ denotes the first coefficient in the series \eqref{eq:diriseri}.

To prepare for the proof of Theorem~\ref{thm:nomoreext}, we consider the vertical translation
\[V_\tau f(s) = f(s+i\tau).\]
The \emph{vertical limit functions} of a Dirichlet series $f$ in $\mathfrak{D}$ are the functions which can be obtained as uniform limits of sequences of vertical translations $(V_{\tau_k} f)_{k\geq1}$ in $\mathbb{C}_\theta$ for any fixed $\theta>\sigmau(f)$. Recall from \cite[Section~2.3]{HLS1997} that the vertical limit functions of the Dirichlet series \eqref{eq:diriseri} coincide with the Dirichlet series of the form
\[f_\chi(s) = \sum_{n=1}^\infty a_n \chi(n) n^{-s},\]
where $\chi$ is a completely multiplicative function from the natural numbers to the unit circle. 

Certain properties of $f$ are preserved under vertical limits. For instance, Bohr's theorem implies that if $f$ is in $\mathfrak{D}$, then $\sigmau(f)=\sigmau(f_\chi)$ for any $\chi$. A consequence of Rouch\'{e}'s theorem (see e.g.~\cite[Lemma~1]{BP2020}) is that $f_\chi(\mathbb{C}_\theta)=f(\mathbb{C}_\theta)$ for any $\chi$ and any $\theta \geq \sigmau(f)$. However, the abscissa of convergence for $f$ and $f_\chi$ may in general be different (see \cite{HLS1997,Helson1969} or \cite[Chapter~8.4]{QQ2020}).
\begin{proof}[Proof of Theorem~\ref{thm:nomoreext}] 
	We begin with the Riemann zeta function
	\[\zeta(s) = \sum_{n=1}^\infty n^{-s},\]
	which satisfies $\sigmac(\zeta)=\sigmau(\zeta)=1$. A result of Bohr \cite{Bohr1911} (see also~\cite[Chapter~4.5]{QQ2020}) asserts that $\zeta(\mathbb{C}_1) = \mathbb{C}\setminus \{0\}$. By the discussion above, it follows that $\sigmau(\zeta_\chi)=1$ and that $\zeta_\chi(\mathbb{C}_1) = \mathbb{C}\setminus \{0\}$ for any $\chi$. Helson \cite{Helson1969} established that there are $\chi$ such that the Dirichlet series $\zeta_\chi$ converges and does not vanish in the half-plane $\mathbb{C}_{1/2}$. Choosing $f = \zeta_\chi$ for such a $\chi$, we obtain the stated result. 
\end{proof}

\bibliographystyle{amsplain-nodash} 
\bibliography{bohrext}

\providecommand{\bysame}{\leavevmode\hbox to3em{\hrulefill}\thinspace}
\providecommand{\MR}{\relax\ifhmode\unskip\space\fi MR }
\providecommand{\MRhref}[2]{%
  \href{http://www.ams.org/mathscinet-getitem?mr=#1}{#2}
}
\providecommand{\href}[2]{#2}
\begin{thebibliography}{10}

\bibitem{Ahlfors1938}
Lars~V. Ahlfors, \emph{An extension of {S}chwarz's lemma}, Trans. Amer. Math.
  Soc. \textbf{43} (1938), no.~3, 359--364. \MR{1501949}

\bibitem{Bohr1911}
Harald Bohr, \emph{\"{U}ber das {V}erhalten von $\zeta(s)$ in der {H}albebene
  $\sigma > 1$}, Nachr. Ges. Wiss. G\"{o}ttingen (1911), 201--208.

\bibitem{Bohr1913}
Harald Bohr, \emph{\"{U}ber die gleichm\"{a}\ss ige {K}onvergenz
  {D}irichletscher {R}eihen}, J. Reine Angew. Math. \textbf{143} (1913),
  203--211. \MR{1580881}

\bibitem{BP2020}
Ole~Fredrik Brevig and Karl-Mikael Perfekt, \emph{Norms of composition
  operators on the {$H^2$} space of {D}irichlet series}, J. Funct. Anal.
  \textbf{278} (2020), no.~2, 108320, 33. \MR{4030275}

\bibitem{BP2021}
Ole~Fredrik Brevig and Karl-Mikael Perfekt, \emph{A mean counting function for
  {D}irichlet series and compact composition operators}, Adv. Math.
  \textbf{385} (2021), Paper No. 107775, 48. \MR{4252762}

\bibitem{GH1999}
Julia Gordon and H{\aa}kan Hedenmalm, \emph{The composition operators on the
  space of {D}irichlet series with square summable coefficients}, Michigan
  Math. J. \textbf{46} (1999), no.~2, 313--329. \MR{1704209}

\bibitem{HLS1997}
H{\aa}kan Hedenmalm, Peter Lindqvist, and Kristian Seip, \emph{A {H}ilbert
  space of {D}irichlet series and systems of dilated functions in
  {$L^2(0,1)$}}, Duke Math. J. \textbf{86} (1997), no.~1, 1--37. \MR{1427844}

\bibitem{Helson1969}
Henry Helson, \emph{Compact groups and {D}irichlet series}, Ark. Mat.
  \textbf{8} (1969), 139--143. \MR{285858}

\bibitem{QQ2020}
Herv\'{e} Queff\'{e}lec and Martine Queff\'{e}lec, \emph{Diophantine
  approximation and {D}irichlet series}, second ed., Texts and Readings in
  Mathematics, vol.~80, Hindustan Book Agency, New Delhi; Springer, Singapore,
  2020. \MR{4241378}

\bibitem{QS2015}
Herv\'{e} Queff\'{e}lec and Kristian Seip, \emph{Approximation numbers of
  composition operators on the {$H^2$} space of {D}irichlet series}, J. Funct.
  Anal. \textbf{268} (2015), no.~6, 1612--1648. \MR{3306358}

\bibitem{Ritt1928}
J.~F. Ritt, \emph{On {C}ertain {P}oints in the {T}heory of {D}irichlet
  {S}eries}, Amer. J. Math. \textbf{50} (1928), no.~1, 73--86. \MR{1506655}

\bibitem{Titchmarsh1986}
E.~C. Titchmarsh, \emph{The theory of the {R}iemann zeta-function}, second ed.,
  The Clarendon Press, Oxford University Press, New York, 1986, Edited and with
  a preface by D. R. Heath-Brown. \MR{882550}

\end{thebibliography}

\end{document}